\documentclass{amsart}
\usepackage{amsmath}

\usepackage{amssymb}
\usepackage{amsthm}
\usepackage{amscd,amsbsy}
\usepackage{xcolor}
\usepackage{tikz-cd}
\usetikzlibrary{arrows,matrix}
\usepackage{float}
\usepackage{mathtools}
\usepackage{enumitem}

\makeatletter
\newcommand{\xdashrightarrow}[2][]{\ext@arrow 0359\rightarrowfill@@{#1}{#2}}
\newcommand{\xdashleftarrow}[2][]{\ext@arrow 3095\leftarrowfill@@{#1}{#2}}
\newcommand{\xdashleftrightarrow}[2][]{\ext@arrow 3359\leftrightarrowfill@@{#1}{#2}}
\def\rightarrowfill@@{\arrowfill@@\relax\relbar\rightarrow}
\def\leftarrowfill@@{\arrowfill@@\leftarrow\relbar\relax}
\def\leftrightarrowfill@@{\arrowfill@@\leftarrow\relbar\rightarrow}
\def\arrowfill@@#1#2#3#4{%
  $\m@th\thickmuskip0mu\medmuskip\thickmuskip\thinmuskip\thickmuskip
   \relax#4#1
   \xleaders\hbox{$#4#2$}\hfill
   #3$%
}
\makeatother

\newtheorem{theorem}{Theorem}[section]
\newtheorem{lemma}[theorem]{Lemma}

\newtheorem{corollary}[theorem]{Corollary}

\theoremstyle{definition}
\newtheorem{definition}[theorem]{Definition}

\newtheorem{example}[theorem]{Example}
\newtheorem{examples}[theorem]{Examples}
\newtheorem{definitions and remarks}[theorem]{Definitions and Remarks}

\theoremstyle{remark}

\newtheorem{remark}[theorem]{Remark}

\numberwithin{equation}{section}

\newcommand{\inv}{\mathrm{inv}}

\newcommand{\oinv}{\overline{\inv}}
\newcommand{\omu}{\overline{\mu}}
\newcommand{\oJ}{\overline{J}}
\newcommand{\oa}{\overline{a}}

\newcommand{\cosupp}{\mathrm{cosupp}\,}

\newcommand{\ord}{\mathrm{ord}}

\newcommand{\codim}{\mathrm{codim}\,}

\newcommand{\adj}{\mathrm{adj}}

\newcommand{\Der}{\mathrm{Der}}

\newcommand{\Sub}{\mathrm{Sub}}
\newcommand{\Bl}{\mathrm{Bl}}

\newcommand{\al}{{\alpha}}
\newcommand{\be}{{\beta}}
\newcommand{\de}{{\delta}}

\newcommand{\De}{{\Delta}}
\newcommand{\ga}{{\gamma}}
\newcommand{\Ga}{{\Gamma}}

\newcommand{\La}{{\Lambda}}

\newcommand{\p}{{\partial}}
\newcommand{\s}{{\sigma}}

\newcommand{\IN}{{\mathbb N}}
\newcommand{\IP}{{\mathbb P}}
\newcommand{\IQ}{{\mathbb Q}}
\newcommand{\IA}{{\mathbb A}}

\newcommand{\IK}{{\mathbb K}}

\newcommand{\cC}{{\mathcal C}}
\newcommand{\cD}{{\mathcal D}}
\newcommand{\cE}{{\mathcal E}}

\newcommand{\cG}{{\mathcal G}}

\newcommand{\cI}{{\mathcal I}}
\newcommand{\cJ}{{\mathcal J}}

\newcommand{\cM}{{\mathcal M}}

\newcommand{\cO}{{\mathcal O}}

\newcommand{\cR}{{\mathcal R}}

\newcommand{\cT}{{\mathcal T}}

\newcommand{\ou}{\overline{u}}

\newcommand{\og}{\overline{g}}

\newcommand{\tilh}{{\widetilde h}}

\newcommand{\tH}{{\widetilde H}}

\newcommand{\ucG}{\underline{\cG}}

\newcommand{\ucI}{\underline{\cI}}
\newcommand{\ucJ}{\underline{\cJ}}
\newcommand{\ucC}{\underline{\cC}}
\newcommand{\ucD}{\underline{\cD}}
\newcommand{\ucE}{\underline{\cE}}
\newcommand{\ucM}{\underline{\cM}}

\newcommand{\ucR}{\underline{\cR}}
\newcommand{\ucT}{\underline{\cT}}

\newcommand{\RN}[1]{%
  \textup{\uppercase\expandafter{\romannumeral#1}}%
}

\begin{document}

\bigskip
\title[Effective resolution of singularities]{Effective resolution of singularities}
\author[E.~Bierstone]{Edward Bierstone}
\author[D.~Grigoriev]{Dima Grigoriev}
\author[P.~Milman]{Pierre D. Milman}
\author[J.~W{\l}odarczyk]{Jaros{\l}aw W{\l}odarczyk}
\address[E.~Bierstone]{University of Toronto, Department of Mathematics, 40 St.~George Street, Toronto, ON, Canada M5S 2E4}
\email{bierston@math.utoronto.ca}
\address[D.~Grigoriev]{CNRS, Laboratoire des Math\'ematiques, Universit\'e de Lille I, 59655 Villeneuve d'Ascq, France}
\email{dmitry.grigoryev@univ-lille.fr}
\address[P.~Milman]{University of Toronto, Department of Mathematics, 40 St.~George Street, Toronto, ON, Canada M5S 2E4}
\email{milman@math.utoronto.ca}
\address[J.~W{\l}odarczyk]{Department of Mathematics, Purdue University, West Lafayette, Indiana 47907, U.S.A.}
\email{wlodarcz@purdue.edu}
\thanks{Reseach supported by NSERC Discovery Grant RGPIN-2017-06537 (Bierstone) and Simons Foundation
Grant MPS-TSM-00008103 (W{\l}odarczyk)}
\thanks{The authors are grateful to Caucher Birkar for suggesting that we write the article for this special volume}
\date{\today}

\begin{abstract}
Consider a projective variety $X \subset \IP^n$ (over an algebraically closed field of characteristic zero),
together with a (reduced) simple normal crossings divisor $E \subset \IP^n$, where
the degrees of both $X$ and $E$ are at most $d$. We show
there is a pair $(n',d')$ which can be explicitly computed in terms of $(n,d)$, such that $(X,E)$ has a log resolution of singularities
$(X',E')$, where $(X',E')$ can be embedded in $\IP^{n'}$ and both $X'$ and $E'$ have degrees at most $d'$ in $\IP^{n'}$. 
\end{abstract}

\maketitle

\setcounter{tocdepth}{1}
\tableofcontents

\section{Introduction}\label{sec:intro}

The \emph{degree} $\deg X$ of an irreducible projective variety $X \subset \IP^n$ (over an algebraically closed field $\IK$)
denotes the number of intersection points of $X$ with a generic linear subspace $L$ of complementary dimension ($\dim L =
n - \dim X$). In general, the \emph{degree} of $X \subset \IP^n$ is the sum of the degrees of the irreducible components of $X$
of all dimensions. The \emph{degree} of a divisor $E \subset \IP^n$ likewise means the sum of the degrees of its (reduced) components (which
are not necessarily irreducible).

For example, if $X$ is a hypersurface and $F$ generates the ideal of homogeneous polynomials vanishing on $X$, then
$\deg X$ equals the degree of the homogeneous polynomial $F$.

The purpose of this article is to give a proof of the following result that Caucher Birkar posed to us as a question.

\begin{theorem}\label{thm:deg}
Consider a projective variety $X \subset \IP^n$ (over an algebraically closed field $\IK$ of characteristic zero),
together with a (reduced) simple normal crossings (snc) divisor $E \subset \IP^n$.
Assume that the degrees of both $X$ and $E$ are at most $d$.
Then there is a pair $(n',d')$ which can be explicitly computed in terms of $(n,d)$, such that $(X,E)$ has a log resolution of singularities
$(X',E')$, where $(X',E')$ can be embedded in $\IP^{n'}$ and both $X'$ and $E'$ have degrees at most $d'$ in $\IP^{n'}$. 
\end{theorem}

An snc divisor $E$ transforms by a blowing-up $\s$ with smooth centre $C$ that is snc with $E$, to an snc divisor defined
by the strict (or birational) transform of $E$ plus the exceptional divisor of $\s$.
\emph{Log resolution} of $(X,E)$ means a resolution of singularities $X' \to X$ given by a composite of smooth blowings-up
as above, such that $X', E'$ have only simple normal crossings, where $E'$ is the final transform of $E$.
If we assume that $E$ is ordered, then $E'$ is also ordered following the sequence of blowings-up.

We prove, in fact, the following variant of Theorem \ref{thm:deg}; the two assertions are equivalent because of the degree
bounds in Section \ref{sec:deg} below

\begin{theorem}\label{thm:gendeg}
Consider a projective variety $X \subset \IP^n$ (over an algebraically closed field $\IK$ of characteristic zero),
together with an snc divisor $E \subset \IP^n$.
We assume that $X$ and $E$
can be defined by homogeneous polynomials of degrees at most $d$. Then there is a pair
$(n',d')$ which can be explicitly computed in terms of $(n,d)$, such that $(X,E)$ has a log resolution of singularities
$(X',E')$, where $(X',E')$ can be embedded in $\IP^{n'}$ and defined in the latter by homogeneous polynomials
of degrees at most $d'$. 
\end{theorem}

Theorem \ref{thm:gendeg} is a consequence of estimates given in \cite{BGMW} or in Sections \ref{sec:affmarkedideal}--\ref{sec:summaryest}
below, in the context of log resolution of
singularities of an embedded algebraic variety $X \hookrightarrow Z$ ($Z$ smooth) together with an ordered snc divisor $E \subset Z$.
The triple $(Z,X,E)$ can be defined locally by polynomial data in affine spaces $\IA^n$ with rational transition mappings,
and there is a log resolution of singularities $(X',E') \subset Z'$ which can be defined by data with effective bounds
on local affine embedding dimension, degrees of all polynomials involved, number of affine charts, number of blowings-up
needed, etc., in terms of bounds on local data needed to define $(X,E) \subset Z$. 
Example \ref{ex:proj} below shows how to apply these effective bounds to obtain Theorem \ref{thm:gendeg}.

The estimates in the article show that the contribution of the dimension $n$ to $(n',d')$ in the theorems above
dwarfs that of the degree $d$. This is highlighted in \cite{BGMW} by a complexity bound in terms of
Grzegorczyk complexity classes $\cE^l$, $l\geq 0$, of primitive recursive (integer) functions, 
where the functions in each $\cE^l$ require at most $l$ nested primitive recursions \cite{Grz}, \cite{WW}. The number of nested
recursions involved in the desingularization algorithm for $(X,E) \subset \IP^n$, is bounded by $n + 3$ (cf.
Remark \ref{rem:summaryremark}).

We use the algorithm for functorial resolution of singularities as presented in \cite{BMinv}, \cite{BMfunct} (the version
in \cite{Wlodar} was used in \cite{BGMW}). Log resolution of singularities
of an embedded pair $(X,E) \subset Z$, or of an ideal $\cI \subset \cO_Z$ together with an snc divisor $E\subset Z$,
follows from resolution of singularities of a collection of resolution data called a \emph{marked ideal}. (Non-embedded)
desingularization of a pair $(X,E)$ follows from functoriality of the embedded case (applied locally). 

There are several articles in the literature on implementation of algorithms for resolution of singularities (e.g., \cite{BS},
\cite{F-KP}). Sections \ref{sec:affmarkedideal}-\ref{sec:bounds} below can be compared with explicit computation of marked ideals 
in \cite{BS}, which raises the challenge of ``super-exponential growth of exponents'' for generators of coefficient marked ideals.
The purpose of our methods is to provide effective bounds on measures of the complexity of the algorithm.

\section{Degree bounds}\label{sec:deg}
The degree bounds in this section were given implicitly in \cite{C}, \cite{G} and \cite{H}.

Let $F_1,\dots, F_r \in \IK[x_0, x_1,\ldots,x_n]$ denote homogeneous polynomials, where $\IK$ is an algebraically closed field 
(of arbitrary characteristic). Let $V(F_1,\ldots, F_r) \subset \IP^n$ denote the projective algebraic variety of solutions of the 
system of equations $F_1=\cdots=F_r=0$. 

\begin{theorem}\label{thm:bezout}
If $\deg F_i \leq d$, $i=1,\ldots,r$, then $\deg V(F_1,\ldots, F_r) \leq d^n$.  
\end{theorem} 

\begin{proof}
We construct trees $T_i$, $i=0,1,\ldots,r$, of irreducible (components of) varieties in $\IP^n$, by recursion on $i$. 
Each tree has a root. The distance $l(v)$ of a vertex $v$ of a tree from its root is called the \emph{level} of $v$. 
An irreducible variety $W_v\subset \IP^n$ with $\dim W_v = n-l(v)$ is attached to $v$ (so we sometimes identify $v$ with $W_v$). 
In particular, the depth of each tree is $\leq n$. 

As a base of the recursion, $T_0$ consists of a single vertex (the root) to which the variety $\IP^n$ is attached. For the recursive 
step, assume that $T_i$ has already been constructed. Given a leaf $v$ of $T_i$, we consider two cases:
If $F_{i+1}$ does not vanish identically on $W_v$, then we take the irreducible components of the variety $W_v\cap \{F_{i+1}=0\}$ 
as the children of $v$ in $T_{i+1}$. On the other hand, if  $F_{i+1}$ vanishes identically on $W_v$, then we keep $v$ as a leaf of   $T_{i+1}$.

This completes the recursive construction of  $T_i$, $i=0,\ldots,r$, and the theorem is a consequence of Lemma \ref{lem:tree} following
in the case $i=r$, $l = n$.
\end{proof}

\begin{lemma}\label{lem:tree}
\emph{(1)} For every $i=0,\ldots,r$, the union of the irreducible varieties attached to the leaves of $T_i$ coincides with $V(F_1,\ldots,F_i)$.

\smallskip
\emph{(2)} For every $i=0,\ldots,r$ and $l = 0,\ldots,n$, the sum of the degrees $\deg W_v$ over all vertices $v$ of $T_i$ of level $l(v)\leq l$,
such that either $l(v)=l$ or $v$ is a leaf of $T_i$, is $\leq d^l$. 
\end{lemma}

\begin{proof} (1) follows directly by induction on $i$.

We prove (2) by induction on $l$. Assume that the statement has already been established for given $l$. 

Then, for every vertex $v$ of $T_i$ of level $l(v)=l$, such that $v$ is not a leaf in $T_i$, the children of $v$ in $T_i$ (which have level $l+1$) 
are the irreducible components of the variety $W_v\cap \{F_h=0\}$, for some $h = 1,\ldots, i$, by the construction of $T_i$. 
The Bezout inequality implies that 
\begin{equation*}
\deg W_v\cap \{F_h=0\} \leq d\cdot \deg W_v. 
\end{equation*}

On the other hand, for each $v$ with $l(v)\leq l$, such that $v$ is a leaf of $T_i$, the contribution of $\deg W_v$ to the sum in (2) 
does not change when passing from $l$ to $l+1$. 
\end{proof}

To give a degree bound in the direction converse to Theorem \ref{thm:bezout}, we use the following lemma of
\cite{H} (see also \cite[Lemma 2.8]{G}).

\begin{lemma}\label{lem:cut}
Consider varieties $V\subset U\subset \IP^n$, where $\deg V\leq d$. Then there exists a homogeneous polynomial 
$F\in \IK[x_0,\ldots,x_n]$ of degree $\leq d$, such that $F$ vanishes on $V$ and, for any irreducible component $U_0$ of $U$ 
that is not a subset of $V$,
\begin{equation*}
\dim (U_0\cap \{F=0\})= \dim U_0 -1.
\end{equation*}
\end{lemma}

\begin{theorem}\label{thm:degree}
Let $V\subset \IP^n$ denote a variety with $\deg V\leq d$. Then there exist homogeneous polynomials $F_0,\ldots, F_n \in \IK[x_0,\ldots,x_n]$ 
with $\deg F_i \leq d$, $i=0,\ldots,n$, such that $V=V(F_0,\ldots,F_n)$.
\end{theorem}

\begin{proof}
We obtain the polynomials $F_i$ by recursion, using Lemma \ref{lem:cut}.
For the base of the recursion, we take $U = \IP^n$. Assume that we have already obtained $F_0,\ldots,F_i$
such that $V\subset V(F_0,\ldots,F_i)$ and, for any irreducible component $U_0$ of $V(F_0,\ldots,F_i)$ that is not a subset of $V$, 
we have $\dim U_0 \leq n-i-1$. Then we apply Lemma \ref{lem:cut} to $U := V(F_0,\ldots,F_i)$ to get $F_{i+1}$. 
\end{proof}

\section{Marked ideals}\label{sec:markedideals}

A \emph{marked ideal} is a quintuple
\begin{equation*}
\ucI = (Z,X,E,\cI,\mu),
\end{equation*}
where
$X\subset Z$ are smooth varieties over an algebraically closed field $\IK$ of characteristic zero, 
$E$ is an ordered snc divisor in $Z$, transverse to $X$ if $\dim X < \dim Z$, $\cI \subset \cO_X$ is an ideal,
and $\mu$ is a nonnegative integer. We will sometimes write $(\cI,\mu)$ as shorthand for $\ucI = (Z,X,E,\cI,\mu)$.
 See \cite[Section 4]{BMfunct}.

We define the \emph{cosupport} of $\ucI$ as
\begin{equation}\label{eq:cosupp}
\cosupp \ucI := \{x \in X:\, \ord_x\cI \geq \mu\}\,.
\end{equation}
We say that $\ucI$ is of \emph{maximal
order} if $\mu = \max\{\ord_x \cI: x \in \cosupp \ucI\}$. The \emph{dimension}
$\dim \ucI$ denotes $\dim X$.

A blowing-up $\s: Z'\to Z$ (with smooth centre $C$) is
\emph{$\ucI$-admissible} (or simply \emph{admissible}) if
$C \subset \cosupp \ucI$, and
$C$, $E$ are snc.
The \emph{(controlled) transform} of $\ucI$ by an admissible blowing-up
$\s: Z' \to Z$ is the marked ideal $\ucI' = (Z',X',E',\cI',\mu'=\mu)$,
where $X'$ is the strict transform of $X$ by $\s$,
$E'$ is the strict transform of $E$ plus the exceptional divisor 
$\s^{-1}(C)$ of $\s$ (introduced as the last component of $E'$),
and $\cI' := \cI_{\s^{-1}(C)}^{-\mu}\cdot \s^*(\cI)$ (where $\cI_{\s^{-1}(C)}
\subset \cO_{X'}$ denotes the ideal of $\s^{-1}(C)$).

In this definition, note that $\s^*(\cI)$
is divisible by $\cI_{\s^{-1}(C)}^\mu$ and $E'$ is an ordered snc 
divisor transverse to $X'$, because $\s$ is admissible. We likewise define
the transform by a sequence of admissible blowings-up.

A \emph{resolution of singularities} of a marked ideal $\ucI = (Z,X,E,\cI,\mu)$
is a sequence of admissible blowings-up after which $\cosupp \ucI'
= \emptyset$. 

For details of the technology of marked ideals (in particular, the notions of sum and product of marked ideals
which will be used below),
and the algorithm for resolution of singularities, see \cite{BMfunct}.
Below we will only briefly recall the notions we will need, to set up a framework for the effective estimates in
the following sections.

\begin{examples}\label{ex:princ,desing}
(1) Principalization (log resolution of singularities) of a pair $(\cI, E)$, where $\cI \subset \cO_Z$ is an ideal
and $E$ is an snc divisor on $Z$, follows from resolution of singularities of the marked ideal $\ucI = (Z,Z,E,\cI,1)$.
%

\medskip\noindent
(2) Desingularization of an embedded variety $Y \subset Z$ with snc divisor $E$ follows from (1) with $\cI = \cI_Y$. 
We can reduce to the case that $E=\emptyset$ by replacing $X$ by $X\cup E$. In
principalization of $\cI_Y$, the strict transform $Y'$ of $Y$ at some step coincides with the centre of blowing-up,
so at this step $Y'$ is smooth and snc with $E'$.
%
\end{examples}

In such applications of desingularization of a marked ideal, the subvariety $X$ in $\ucI = (Z,X,E,\cI,\mu)$ arises in 
the inductive construction as a maximal contact subvariety.

In Example \ref{ex:princ,desing}\,(2), if $Y$ is a hypersurface, then the centre of each blowing-up determined by the algorithm
will lie in the maximum order locus of (the corresponding strict transform of) $Y$. In the general case, resolution
of singularities of the marked ideal $\ucI = (Z,Z,E,\cI_Y,1)$ leads to a somewhat weaker version of embedded
desingularization, where the centres of blowing-up are smooth in the successive transforms of $Z$, but do not
necessarily have smooth intersections with the transforms of $Y$ \cite[\S8.2]{BMfunct}. For the stronger version of embedded
resolution of singularities, see \cite{BMfunct}.

\begin{example}\label{ex:proj} \emph{Case of a projective variety and snc divisor $(X,E) \subset \IP^n$.} We begin with
a marked ideal $\cI$ as in Example \ref{ex:princ,desing}(2), with $Z=\IP^n$, and we can reduce to the case
that $E=\emptyset$ by replacing $(X,E)$ by $(X\cup E,\,\emptyset)$. Then there is an \emph{affine marked ideal} $\ucT$
(a certain collection of resolution data associated to $\ucI$) as in 
Section \ref{sec:affmarkedideal} below, where the affine charts
$(\IK^{n_\al})_\al$ are affine charts of $\IP^n$ (each $n_\al = n$). See Definition \ref{def:affmarkedideal} and Remark 
\ref{rem:affmarkedideal}. In the initial data given by Definition \ref{notn:affmarkedideal}, all terms are bounded in a
simple way in terms of the given degree bounds of Theorem \ref{thm:gendeg}.

If $\s: Z' \to Z$ is a blowing-up of $Z$ with smooth closed centre $C$, then $Z'$ is embedded in $\IP^n\times\IP^r$,
$r = \codim C - 1$, and the affine marked ideal $\ucT'$ defined over $\ucT$ by the blowing-up (Section \ref{sec:affmarkedideal})
has the affine charts of $\IP^n\times\IP^r$; $\ucT'$ corresponds to a marked ideal on the strict transform $Z'$ of $Z$
in $\IP^n\times\IP^r$.

The final affine marked ideal $\ucT_*$ given by the resolution algorithm corresponds to a marked ideal $\ucI_*$
on the final strict transform of $Z$ in a finite product $\IP^n\times\IP^{r_1}\times\IP^{r_2}\times\cdots$, where the codimensions
of the successive centres of blowing up are $r_1+1,\, r_2+1,\ldots$ in the successive products of projective spaces.
We can embed the final product in $\IP^{n'}$, for suitable $n'$, using the Segre embedding, and effective estimates
for $(X',E')$ follow from those of Section \ref{sec:summaryest}.
\end{example}

\section{Algorithm for functorial resolution of singularities of a marked ideal}\label{sec:algorithm}
We sketch the proof of desingularization of a marked ideal which can be found in detail in \cite[Sections 5--7]{BMfunct}.

\begin{theorem}\label{thm:alg} 
Let $\ucI = (Z,X,E,\cI,\mu)$ denote a marked ideal, where $\mu>0$.
Then $\ucI$ admits a resolution of singularities 
\begin{equation*}
Z = Z_0 \stackrel{\s_1}{\longleftarrow} Z_1 \longleftarrow \cdots
\stackrel{\s_{t}}{\longleftarrow} Z_{t}\,,
\end{equation*}
such that, if $\ucI_j$ denotes the $j$'th transform of $\ucI$, then the 
following properties are satisfied.
\begin{enumerate}
\item
There are Zariski upper-semicontinuous functions $\inv_{\ucI}$, $\mu_{\ucI}$
and $J_{\ucI}$ defined on $\cosupp \ucI_j$, for all $j$, where 
\begin{description}
\item[$\inv_{\ucI}$] takes values in the set of sequences 
$$
(\nu_1,s_1,\nu_2,s_2,\ldots,\nu_q,s_q, \nu_{q+1}),
$$
ordered lexicographically, where $\nu_1,\ldots,\nu_q$ are positive rational numbers,
each $s_k$ is a nonnegative integer, and $\nu_{q+1} = 0$ or $\infty$;
\item[$\mu_{\ucI}$] takes values in $\IQ_{\geq 0} \cup \{\infty\}$;
\item[$J_{\ucI}$] takes values in the set of subsets $I$ of $E_j$, for all $j$, ordered
as follows: If $E_j = \{H_1,\ldots,H_r\}$ ($r \geq j$), then we associate to $I \subset E_j$
the lexicographic order of $(\de_1,\ldots,\de_r)$, where $\de_k = 0$ if $H_k\notin I$
and $\de_k = 1$ if $H_k\in I$.
\end{description}
\item
Each centre of blowing up $C_j \subset Z_j$ is given by the maximum locus
of $(\inv_{\ucI}, J_{\ucI})$ on $\cosupp \ucI_j$ (where pairs are ordered
lexicographically).
\item
Let $a \in \cosupp \ucI_{j+1}$ and $b = \s_{j+1}(a)$. If $b \not\in C_j$, then
$$
\inv_{\ucI}(a)=\inv_{\ucI}(b),\quad \mu_{\ucI}(a)=\mu_{\ucI}(b),\quad
J_{\ucI}(a)=J_{\ucI}(b).
$$
If $b \in C_j$, then
$$
\left(\inv_{\ucI}(a),\mu_{\ucI}(a)\right) 
< \left(\inv_{\ucI}(b),\mu_{\ucI}(b)\right).
$$
\item
The functions in (1) above are invariants of the equivalence class of $\ucI$
and smooth morphisms (see below).
\end{enumerate}
\end{theorem}

Our construction of the desingularization invariant $\inv_{\ucI}$ (in particular, the terms
$\nu_k$ in $\inv_{\ucI}$, and $\mu_{\ucI}$) is based
on the following.

\begin{definition}\label{def:invts} {Invariants of a marked ideal.} Given a marked
ideal $\ucI = (Z,X,E,\cI,\mu)$ and a point $a \in \cosupp \ucI$, we set
\begin{equation}\label{eq:invts}
\mu_a(\ucI) := \frac{\ord_a\cI}{\mu} \quad \mbox{and} \quad 
\mu_{H,a}(\ucI) := \frac{\ord_{H,a}\cI}{\mu},\,\, H \in E,
\end{equation}
where $\ord_{H,a}\cI$ denotes the \emph{order} of $\cI \subset \cO_X$
\emph{along} $H|_X$ at $a$;\, i.e., the largest $\rho \in \IN$ such that
$\cI_a \subset \cI_{H|_X,a}^\rho$.
\end{definition}

\begin{theorem}[{\cite[Section 6]{BMfunct}}]\label{thm:invts}
Both $\mu_a(\ucI)$ and $\mu_{H,a}(\ucI)$ depend only on the equivalence 
class of $\ucI$ and $\dim X$.
\end{theorem}

See \cite[Definitions 2.5]{BMfunct} for a formal definition of \emph{equivalence} of two marked ideals $\ucI, \ucJ$ on
$Z$. Two equivalent marked ideals on $Z$ have the property that a sequence of blowings-up
is admissible for one if and only if it is admissible for the other; i.e., the two marked ideals have the same
sequences of admissible blowings-up. More precisley, two marked ideals $\ucI,\, \ucJ$
on $Z$ are \emph{equivalent} if $\ucI|_U$ and $\ucJ|_U$ have the same sequences of \emph{test transformations}, 
for every open subset $U$ of $Z$, where test transformations are either admissible blowings-up or one of two
other simple kinds of transformations of marked ideals (see \cite[Section 2]{BMfunct}). We do not give details here
because the notion is not needed for the effective estimates in these notes. It is used for functoriality properties
of the desingularization algorithm in \cite{BMinv}, \cite{BMfunct}; in particular, to show that the centres of blowing up
defined locally on different maximal contact subvarieties glue together to define global blowings-up

\begin{proof}[Sketch of the proof of Theorem \ref{thm:alg}] 
Let $\ucI = (Z,X,E,\cI,\mu)$ denote a marked ideal, where $\mu>0$. The proof is by
induction on $\dim \ucI := \dim X$.

First suppose $\dim X = 0$. Then $X$ is a discrete set. Consider $a\in X$. Then $a\in \cosupp \ucI$
if and only if $\cI_a=0$ (so that $\ord_a\cI = \infty$). We set
$$
\inv_{\ucI}(a):=\infty,\quad \mu_{\ucI}(a):=\infty,\quad J_{\ucI}(a):=\emptyset.
$$
We blow up with centre given by the discrete set $\cosupp \ucI$ to resolve the singularities of $\ucI$.

The inductive step of the proof breaks up into two independent steps:

\smallskip\noindent
Step I. Maximal order case. Functorial desingularization of a marked ideal of dimension $m-1$ implies 
functorial desingularization of a marked ideal of maximal order, of dimension $m$. Only this step
uses the hypothesis of induction on dimension. 

\smallskip\noindent
Step II. General case. Functorial desingularization of a marked ideal of maximal order of 
dimension $m$ implies functorial
desingularization of a general marked ideal of dimension $m$. This step involves the invariants
of Theorem \ref{thm:invts}.

\medskip\noindent
\emph{Derivative and coefficient ideals.}
The maximal order case Step I involves passing from a marked ideal $\ucI$ of maximal order to
a \emph{coefficient ideal} on a locally defined \emph{maximal contact} hypersurface. Coefficient
ideals are defined using logarithmic derivatives of local sections of $\cI$. Let $\Der_{X,E}$ denote
the ring of derivations (sections of the sheaf of derivations) $D$ on $X$ such that $D(\cI_E|_X)
\subset \cI_E|_X$ (see \cite[Section 3]{BMfunct} and also \S\ref{subsec:der} below).

The \emph{log derivative ideal} $\cD_E(\cI)$ denotes the ideal generated by local sections of $\cI$
and elements of $\Der_{X,E}$ applied to local sections. Given a marked ideal $\ucI = (Z,X,E,\cI,\mu)$,
we define derivative marked ideals
\begin{align*}
\ucD_E(\ucI) &:= (Z,X,E,\cD_E(\cI), \mu-1) = (\cD_E(\cI), \mu-1),\\
\ucD_E^{j+1}(\ucI) &:= \ucD_E(\ucD_E^{j}(\ucI)) = (\cD_E^{j+1}(\cI), \mu-j-1),\quad j=1,\ldots,\mu -1,
\end{align*}
as well as
\begin{equation}\label{eq:C}
\ucC(\ucI) := \sum_{j=0}^{\mu-1} \ucD_E^{j}(\ucI).
\end{equation}
Then $\ucC(\ucI)$ is equivalent to $\ucI$ \cite[Cor.\,3.11]{BMfunct}. The sum in \eqref{eq:C} is
a weighted sum, according to the definition of sum of marked ideals \cite[\S3.3]{BMfunct}.
Write $\ucC(\ucI) = (Z,X,E,\allowbreak \cC(\ucI), \mu_{\ucC(\ucI)})$.

Consider a section $u$ of $\cD_E^{\mu -1}(\cI)$ on $X|_U$ where $U$ is an open subset of $X$. 
If $u$ has maximum order $1$ and is transverse to $E$, then we call $Y := V(u) \subset X|_U$ a
\emph{maximal contact hypersurface} for $\ucI$. If $Y \subset X|_U$ is a maximal contact hypersurface,
then we define the \emph{coefficient marked ideal}
\begin{equation*}
\ucC_Y(\ucI) := (U, Y, E, \cC(\ucI)|_Y, \mu_{\ucC(\ucI)}).
\end{equation*}
The coefficient ideal $\ucC_Y(\ucI)$ is equivalent to $\ucI_U$.

A marked ideal $\ucI$ of maximal order such that 
$E=\emptyset$ admits a maximal contact hypersurface in some neighbourhood of any point
of $\cosupp \ucI$ (cf. \S\ref{subsec:maxcontactbounds} below).

\medskip\noindent
{\bf Step I. Maximal order case.} 
Let $\ucI = (Z,X,E,\cI,\mu)$ denote a marked ideal of maximal order, and let $m := \dim \ucI$.

\smallskip\noindent
{\bf Case A. $E = \emptyset$.} Let $a \in \cosupp \ucI$. Then
there is a hypersurface of maximal contact $Y$ for $\ucI$ at $a$; $Y \subset
X\cap U$, where $U$ is a neighbourhood of $a$ in $Z$. Then $\ucC_Y(\ucI)$ is
equivalent to $\ucI|_U = (U,X\cap U, \emptyset, \cI|_{X\cap U}, \mu)$.
Therefore, a resolution of singularities of
$\ucC_Y(\ucI)$ (which exists by induction) is a resolution of 
singularities of $\ucI|_U$.

Coefficient marked ideals $\ucC_{Y_1}(\ucI)$ and $\ucC_{Y_2}(\ucI)$ defined in overlapping
charts $U_1$ and $U_2$ are equivalent in $U_1\cap U_2$ (since both are equivalent
to $\ucI|_{U_1\cap U_2}$); therefore, by functoriality in dimension $m-1$, their
resolution sequences  are the same over $U_1\cap U_2$ (not counting
blowings-up that restrict to isomorphisms over $U_1\cap U_2$). 

Given $x \in \cosupp \ucI_j$ lying over $U$, we set
$$
\oinv_{\ucI}(x) := \left(0, \inv_{\ucC_Y(\ucI)}(x)\right),\quad
\omu_{\ucI}(x) := \mu_{\ucC_Y(\ucI)}(x),\quad 
\oJ_{\ucI}(x) := J_{\ucC_Y(\ucI)}(x),
$$
where the invariants for $\ucC_Y(\ucI)$ are defined by induction. Semicontinuity of
these invariants tells us which nontrivial centres of blowing up in the various charts $U$ 
should be taken first, etc.

\smallskip\noindent
{\bf Case B. General maximal order case.} Let $\ucI_{\emptyset}$ denote
the marked ideal $(Z,X, \emptyset, \linebreak[0] \cI,\mu)$. Then
$\cosupp \ucI_{\emptyset} = \cosupp \ucI$. Let $a \in \cosupp \ucI$.
Then there exists a
hypersurface of maximal contact $Y$ for $\ucI_{\emptyset}$, defined in a
neighbourhood $U$ of $a$. We introduce $\ucC := \ucC_Y(\ucI_{\emptyset})
= (U,Y,\emptyset, \cC, \mu_{\ucC})$, as in Case A above.

If $x \in X$, set 
$$
s(x) := \#\{H \in E: x \in H\}
$$
(where  $\#$
denotes the cardinality of a finite set). After 
transformation of $\ucI$ by any sequence of admissible blowings-up (or test morphisms),
\begin{equation*}
Z = Z_0 \stackrel{\s_1}{\longleftarrow} Z_1 \longleftarrow \cdots
\stackrel{\s_{t}}{\longleftarrow} Z_{t}\,,
\end{equation*}
we will continue to write $E$ (as opposed to $E_t$)
for the divisor whose components are the strict transforms of those of $E$,
with the same ordering as in $E$, and we will continue to write $s(x)$ for
$\#\{H \in E: x \in H\}$.

Let $s := \max \{s(x): x \in \cosupp \ucI\}$,
and let $\Sub(E,s)$ denote the set of $s$-element subsets of $E$
(i.e., of the set of components of $E$). Let $\ucE$ denote the \emph{boundary} marked ideal
$\ucE := \left(U,Y,\emptyset, \cE, \mu_{\ucC}\right)$, where
\begin{equation}\label{eq:s}
\cE := \prod_{\La \in \Sub(E,s)}\sum_{H \in \La} \cI_H^{\mu_\cC}\cdot \cO_Y
\end{equation}
(the cosupports of the factors in this product are disjoint).
Set $\ucJ := \ucC + \ucE$. 

Clearly, $\cosupp \ucE = \{x \in Y: s(x) \geq s\}$, so that
$$
\cosupp \ucJ = \cosupp \ucI|_U \bigcap \{x \in Y: s(x) = s\},
$$
(and likewise after any sequence of admissible blowings-up of $\ucJ$ or, more generally,
any sequence of test transformations of $\ucJ$, as involved in the definition of equivalence
\cite[Definitions 2.5]{BMfunct}).
Thus any sequence of test transformations of $\ucJ$ is a sequence
of test transformations of $\ucI|_U$, and $s(x)=s$ at every point $x$
of the centre $C$ of each admissible 
blowing-up. Therefore, the equivalence class
of $\ucJ$ depends only on that of $\ucI|_U$, and blowings-up which are 
admissible for $\ucJ$ are also admissible for $\ucI|_U$.

By induction on dimension, there is a resolution of singularities of $\ucJ$.
As in Case A, these local resolutions patch together
to define a sequence of admissible blowings-up of $\ucI$. 

To define the corresponding invariants: given $x \in \cosupp \ucJ$ (or in the cosupport
of its successive transforms) we set
$$
\oinv_{\ucI}(x) := \left(s, \inv_{\ucJ}(x)\right),\quad
\omu_{\ucI}(x) := \mu_{\ucJ}(x),\quad
\oJ_{\ucI}(x) := J_{\ucJ}(x).
$$

If the sequence of blowings-up for $\ucJ$ resolves $\ucI$, then we have finished
Case B. Otherwise, 
$\cosupp \ucI_j$ and $\{x: s(x)=s\}$ become disjoint
for some index $j$ of the blowings-up sequence, say for $j=q_1$ 
(i.e., $s(x) < s$, for all $x \in \cosupp \ucI_j$). In the latter case, we repeat
the process using $\ucC_{q_1} := \ucC_Y(\ucI_{\emptyset})_{q_1}$ in place of
$\ucC$, and using the new value $s_1$ of $\max \{s(x): x \in \cosupp \ucI_{q_1}\}$
in place of $s$. (Here $s(x)$ has the same meaning as above---it is defined
using $E$ as opposed to $E_{q_1}$. The process above is repeated using 
$\ucJ_1 := \ucC_{q_1} + \ucE_1$ in place of $\ucJ$, where $\ucE_1$ is given
by \eqref{eq:s} with $s_1$ in place of $s$.)

We thus get marked ideals $\ucI_{q_1},\ldots,\ucI_{q_k},\ldots$ with
$s > s_1 >\cdots> s_k >\cdots$, and, for each $k$ and $U$, a marked ideal $\ucJ_k
= \ucC_{q_k} + \ucE_k$ analogous to $\ucJ_1$. 
After finitely many steps as above, $\ucI$ is resolved.
(If, after $k$ steps, $\ucI$ is not resolved but $s_k = 0$, then $\ucI$ is
resolved by the next step.)

Note that if $x \in \cosupp \ucI_j$ and $s(x) < s$, then all blowings-up in
the desingularization tower of $\ucI$ are isomorphisms over $x$, until we
reach a year $q_k$ where $s(x)$ is the maximum value. Then the values of
the invariants at $x$ equal their values in year $q_k$.

We have completed Case B and therefore Step I.
\medskip

\noindent
{\bf Step II. General case.} Let $\ucI = (Z,X,E,\cI,\mu)$ be an arbitrary
marked ideal ($\mu>0$).

First suppose that $\cI = 0$ (so that $\cosupp \ucI = X$). Then we can
blow up with centre $X$ to resolve singularities. We set
$$
\inv_{\ucI}(x) := \infty,\quad \mu_{\ucI}(x) := \infty,\quad 
J_{\ucI}(x) := \emptyset,
$$
where $x \in \cosupp \ucI = X$.

Now suppose that $\cI \neq 0$. Write
\begin{equation}\label{eq:monres}
\cI = \cM(\ucI)\cdot \cR(\ucI),
\end{equation}
where $\cM(\ucI)$ is a product of powers of the principal ideals $\cI_H$ of the components $H$
of $E$, and $\cR(\ucI)$ is divisible by no such principal ideal. We call $\cM(\ucI)$ the 
\emph{monomial part} of $\cI$, and $\cR(\ucI)$ the \emph{residual part} of $\cI$.
We consider two cases.
\smallskip

\noindent
{\bf Case A. Monomial case $\cI = \cM(\ucI)$.} Let $a \in \cosupp \ucI$. In a neighbourhood of $a$, we can write
$\cI = \cI_{H_{i_1}}^{\al_1}\cdots \cI_{H_{i_r}}^{\al_r}$
(where $a \in H_{i_1} \cap \cdots \cap H_{i_r}$ and where we write $\cI_{H_{i_k}}$
instead of $\cI_{H_{i_k}}\cdot\cO_X$ to simplify the notation);
in particular,
$\al_1 + \cdots + \al_r > \mu$. By Theorem \ref{thm:invts}, 
$\mu_a(\ucI) = (\al_1 + \cdots + \al_r)/\mu$ is a local invariant 
of the equivalence
class of $\ucI$.) Then (in the neighbourhood above) 
$\cosupp \ucI = \bigcup Y_I$, where each $Y_I := X \cap \bigcap_{H \in I}H$,
and $I$ runs over the \emph{smallest} subsets of $\{H_{i_1},\ldots,H_{i_r}\}$
such that $\sum_{l \in I} \al_l \geq \mu$;
in other words, $I$ runs over the subsets of $\{H_{i_1},\ldots,H_{i_r}\}$
such that
$$
0 \leq \sum_{l \in I} \al_l -\mu < \al_k, \quad \mbox{for all } k \in I.
$$
(We have simplified the notation by identifying subsets of
$\{H_{i_1},\ldots,H_{i_r}\}$ with subsets of $\{1,\ldots,r\}$.)

Let $J_{\ucI}(a)$ denote the maximum of these subsets $I$ (with respect to the
order introduced in Theorem \ref{thm:alg}). Clearly, $J_{\ucI}(x)$ is Zariski upper-semicontinuous on 
$\cosupp \ucI$, and the maximum locus of $J_{\ucI}(x)$ consists of at most one
irreducible component of $\cosupp \ucI$ through each point of the latter. 

Consider the blowing-up $\s$ with centre $C$ given by the maximum locus
of $J_{\ucI}(\cdot)$; If $a \in C$, then (in a neighbourhood as above),
$C = X\cap \bigcap_{l \in J_{\ucI}(a)}H_{i_l}$. We can choose local coordinates
$(x_1,\ldots,x_n)$ for $X$ at $a$ such that, for each $k \in J(a):=J_{\ucI}(a)$,
$x_k$ is a local generator $x_{H_{i_k}}$ of the ideal $\cI_{H_{i_k}}$.
Then, in the $x_k$-chart of the blowing-up $\s$, the transform of $\cI$
is given by
$$
\cI' = \cI_{H_{i_1}'}^{\al_1}\cdots \cI_{\s^{-1}(C)}^{\sum_{J(a)}\al_l-\mu}
\cdots \cI_{H_{i_r}'}^{\al_r},
$$
(with $\cI_{\s^{-1}(C)}^{\sum_{J(a)}\al_l-\mu}$ in the $k$'th place). Since
$\sum_l \al_l -\mu < \al_k$, $\mu_{a'}(\ucI') = \mu_a(\ucI) - p/\mu$, where
$p$ is a positive integer, over all points $a$ of some component of
$\cosupp \ucI$. We therefore resolve singularities after a finite number
of steps.

Given $x \in \cosupp \ucI_j$, we set
$$
\inv_{\ucI}(x) := 0,\quad \mu_{\ucI}(x) := \mu_x(\ucI_j) = 
\frac{\ord_x\cI_j}{\mu},
$$
and we let $J_{\ucI}(x)$ denote the maximum among the subsets of $E_j$
that define the components of $\cosupp \ucI_j$ at $x$.
The centres of blowing up above are given by the successive maximum value loci 
of $(\inv_{\ucI}(x), J_{\ucI}(x))$.
\smallskip

\noindent
{\bf Case B. General case.} We define the \emph{residual order} of $\cI$ at 
$a \in \cosupp \ucI$ as
\begin{equation*}
\nu_{\ucI}(a) := \frac{\ord_a\cR(\ucI)}{\mu}
                      = \mu_a(\ucI) - \sum_{H\in E}\mu_{H,a}(\ucI).
\end{equation*}
According to Theorem \ref{thm:invts},
$\nu_{\ucI}(a)$ depends only on the equivalence class of $\ucI$ and $\dim X$.
Set
$$
\ord\,\cR(\ucI) := \max_{x \in \cosupp \ucI} \ord_x \cR(\ucI),
$$
and define marked ideals 
\begin{align*}
\ucR(\ucI) &:= (Z,X,E, \cR(\ucI), \ord\,\cR(\ucI)),\\ 
\ucM(\ucI) &:= (Z,X,E, \cM(\ucI), \mu-\ord\,\cR(\ucI)).
\end{align*}
We define the \emph{companion ideal} of $\ucI = (\cI,\mu)$ as
$$
\ucG(\ucI) :=
\begin{cases}
\ucR(\ucI) + \ucM(\ucI), &\text{if $\ord\,\cR(\ucI) < \mu$;}\\
\ucR(\ucI),              &\text{if $\ord\,\cR(\ucI) \geq \mu$.}
\end{cases}
$$
(We can treat both cases simultaneously by defining
the cosupport and transforms of $\ucM(\ucI)$ in the case that
$\ord\,\cR(\ucI) \geq \mu$ exactly as in \eqref{eq:cosupp}, even though
$\mu-\ord\,\cR(\ucI) \leq 0$; e.g., $\cosupp \ucM(\ucI) = X$ in this case.)

The purpose of factoring out the monomial part of $\ucI$ \eqref{eq:monres} to get the maximal order
residual ideal $\ucR(\ucI)$ is to reduce to the maximal order case Step I, and eventually to the monomial case.
The companion ideal is constructed to guarantee that the blowings-up involved will be admissible for $\ucI$
(see \cite[Section 5]{BMfunct}). Moreover, it follows from Theorem \ref{thm:invts} that the 
equivalence class of $\ucG(\ucI)$ depends only on that of $\ucI$ and $\dim X$ (see \cite[Cor.\,5.3]{BMfunct}).

Now, we can resolve the singularities of the marked ideal of maximal order
$\ucG := \ucG(\ucI)$, using Step I. The blowings-up involved are admissible for $\ucI$.

To define the invariants: given $x \in \cosupp \ucG_j$, we set
$$
\inv_{\ucI}(x) := \left(\frac{\ord\,\cR(\ucI)}{\mu}, \oinv_{\ucG}(x)\right),\quad
\mu_{\ucI}(x) := \omu_{\ucG}(x),\quad
J_{\ucI}(x) := \oJ_{\ucG}(x).
$$

If $x \in \cosupp \ucI \setminus \cosupp \ucG$, then $\ord_x \cR(\ucI)$
will be the maximum order of $\cR(\ucI)$ in some neighbourhood of $x$,
so we can define the invariants in the same way over such a neighbourhood.

This resolution leads, after a finite number of steps (say $r_1$ steps) to
a marked ideal $\ucI_{r_1}$ such that either $\ord\,\cR(\ucI_{r_1}) <
\ord\,\cR(\ucI)$, or $\cosupp \ucI_{r_1} = \emptyset$. (See \cite[Section 5]{BMfunct}.)
In the latter case, we have resolved the singularities of $\ucI$. In the
former case, we can repeat the process using $\left(\cR(\ucI_{r_1}),
\ord\,\cR(\ucI_{r_1})\right)$. If $x \in \cosupp \ucG(\ucI_{r_1})$ maps to
the complement of $\cosupp \ucG(\ucI)$, then all previous blowings-up are
isomorphisms at the successive images of $x$, so the values of the invariants
at these points are the same as at $x$. In particular, for all $x \in \cosupp \ucI_j$,
the first entry of $\inv_{\ucI}(x)$ is 
$$
\frac{\ord_x\cR(\ucI_j)}{\mu} = \nu_{\ucI_j}(x).
$$

We thus get marked ideals $\ucI_{r_1},\ldots,\ucI_{r_k},\ldots$ such that
$$
\ord\,\cR(\ucI) > \ord\,\cR(\ucI_{r_1}) > \cdots > \ord\,\cR(\ucI_{r_k}) > \cdots .
$$
The process terminates after a finite number of steps, when either we have 
$\ord\,\cR(\ucI_{r_k}) \allowbreak = 0$ (i.e., we have reduced to the monomial case
$\ucI_{r_k} = \ucM(\ucI_{r_k})$), or we have $\cosupp \ucI_{r_k} = \emptyset$
(i.e., we have resolved singularities).

All properties of the invariants in the statement of the theorem follow by induction and
semicontinuity of $\ord_x$.
\end{proof}

\section{Resolution data}\label{sec:affmarkedideal}
A marked ideal can be described by a collection of local polynomial data in affine spaces $\IA^n$ with birational transition mappings
as glueing data.
In this section, we formalize this collection of data as an \emph{affine marked ideal}. Effective estimates on the degrees
and number of blowings-up will be expressed in terms of an affine marked ideal. 

\begin{definition}\label{def:affmarkedideal}
An \emph{affine marked ideal} $\ucT$ is a collection of tuples together with an associated order $\mu$,
\begin{equation}\label{eq:affmarkedideal}
\ucT = \left(\left\{U_{\al\be}, X_{\al\be}, E_{\al\be}, \cI_{\al\be}, \left(\IK^{n_\al}\right)_\al: \al \in A,\,\be\in B_\al\right\}, \mu\right),
\end{equation}
where $A$ and the $B_\al$ are finite index sets, and
\begin{enumerate}
\item $(\IK^{n_\al})_\al \cong \IK^{n_\al}$, with affine coordinates $x_\al = (x_{\al 1},\ldots,x_{\al,n_\al})$;
\item $\{U_{\al\be}: \be\in B_\al\}$ is an open covering of $(\IK^{n_\al})_\al$, where $U_{\al\be} \subset (\IK^{n_\al})_\al$
is the complement of the zero set of a polynomial $f_{\al\be} \in \IK[x_\al]$;
\item $E_{\al\be}$ is a collection of smooth divisors in $(\IK^{n_\al})_\al$, each given by an equation $x_{\al j} = 0$,
for some $j=1,\ldots,n_\al$;
\item $X_{\al\be} \subset (\IK^{n_\al})_\al$ is a closed subset, where $X_{\al\be} \cap U_{\al\be}$ is smooth; moreover,
there is a set of parameters (coordinates) on $U_{\al\be}$,
$$
u_{\al\be, 1},\ldots,u_{\al\be,n_\al} \in \IK[x_\al],
$$
where each $u_{\al\be, i}$ is either a coordinate $x_{\al j}$ describing an exceptional divisor (i.e., an element of $E_{\al\be}$)
or is transverse to $E_{\al\be}$ over $U_{\al\be}$, 
and $\cI_{X_{\al\be}}$ is the ideal
$$
\cI_{X_{\al\be}} = (u_{\al\be, 1},\ldots,u_{\al\be,n_\al-m}) \subset \IK[x_\al],
$$
with $u_{\al\be, 1},\ldots,u_{\al\be,n_\al-m}$ all transverse to $E_{\al\be}$;
\item $\cI_{\al\be}$ is an ideal $(g_{\al\be,1},\ldots,g_{\al\be,\overline{j}}) \subset \IK[x_\al]$,
and
$$
\cosupp (\cI_{\al\be},\mu) \cap U_{\al\be} \cap U_{\al\be'} = \cosupp (\cI_{\al\be'},\mu) \cap U_{\al\be} \cap U_{\al\be'},
$$
where $(\cI_{\al\be},\mu)$ denotes the marked ideal $(U_{\al\be},X_{\al\be}\cap U_{\al\be}, E_{\al\be}, \cI_{\al\be},\mu)$;
\item for all $\al_1,\al_2 \in A$, $\be_1\in B_{\al_1}$, $\be_2\in B_{\al_2}$,
there exist 
\begin{align*}
v_{\al_1\be_1\al_2\be_2,1},\ldots,v_{\al_1\be_1\al_2\be_2,n_{\al_2}} &\in \IK[x_{\al_1}],\\
w_{\al_1\be_1\al_2\be_2,1},\ldots,w_{\al_1\be_1\al_2\be_2,n_{\al_2}} &\in \IK[x_{\al_1}],
\end{align*}
such that
$$
x_{\al_1} \mapsto \left(\frac{v_{\al_1\be_1\al_2\be_2,1}}{w_{\al_1\be_1\al_2\be_2,1}},\ldots,\frac{v_{\al_1\be_1\al_2\be_2,n_{\al_2}}}{w_{\al_1\be_1\al_2\be_2,n_{\al_2}}}\right)(x_{\al_1})
$$
induces a birational mapping $i_{\al_1\be_1\al_2\be_2}: X_{\al_1\be_1} \dashrightarrow X_{\al_2\be_2}$;
moreover, the birational mappings $i_{\al_1\be_1\al_2\be_2}$ determine a variety $X_{\ucT}$, unique up to isomorphism,
together with open embeddings $j_{\al\be}: X_{\al\be} \cap U_{\al\be} \hookrightarrow X_{\ucT}$ defining an open covering
of $X_{\ucT}$, such that $j_{\al_2\be_2}^{-1}\circ j_{\al_1\be_1} = i_{\al_1\be_1\al_2\be_2}$;
\item $\mu$ is a nonnegative integer.
\end{enumerate}
\end{definition}

\begin{remark}\label{rem:affmarkedideal}
A marked ideal $\ucI = (Z,X,E,\cI,\mu)$
determines an affine marked ideal $\ucT$ as above, such that $X_{\ucT} = X$ (using local embeddings of $Z$ in affine spaces), 
at least in the case that $E=\emptyset$ (cf. Example \ref{ex:proj}). The marked ideals then introduced successively following the
desingularization algorithm correspond the successive affine marked ideals described below. 

If $E\neq \emptyset$, the associated affine marked ideal $\ucT$ may not satisfy the condition (3) in Definition 
\ref{def:affmarkedideal}. The exceptional divisors of the successive blowings-up, however, can be described in a way that they satisfy (3).
In general, we can perform Step IB of Section \ref{sec:algorithm} (with $\cI$ in place of $\cC$) as a preliminary
step, before beginning the general desingularization algorithm (which starts with Step II; cf. Section \ref{sec:summaryest}).
This preliminary application of Step IB has the effect of moving $E$ away from $\cosupp \cI$, and the new divisors introduced
in the process will satisfy condition (3).

We could also include additional terms $Z_{\al\be}$ in $\ucT$
corresponding to $Z$, where $\cI_{Z_{\al\be}} = (u_{\al\be 1},\ldots,u_{\al\be,n_\al-p})$, $p\geq m$. This will not change
the effective estimates below because the generators of $\cI_{Z_{\al\be}}$ are already included among those of $\cI_{X_{\al\be}}$.
Moreover, $U_{\al\be}$ in $\ucT$ plays the part of $Z$ in $\ucI$.

If we do include $Z_{\al\be}$ in $\ucT$, then we can recover $\ucI$ from $\ucT$, using $Z = Z_{\ucT}$, analogous to $X_{\ucT}$.
\end{remark}

\begin{definition}\label{def:cosupp}
The \emph{cosupport} $\cosupp \ucT$ denotes the set of cosupports of the marked ideals
$$
(U_{\al\be},\, X_{\al\be} \cap U_{\al\be},\, E_{\al\be} \cap U_{\al\be},\, \cI_{\al\be}|_{X_{\al\be} \cap U_{\al\be}},\, \mu),\quad
\al \in A,\, \be \in B_\al.
$$
\end{definition}

\begin{definition}\label{notn:affmarkedideal}
Given an affine marked ideal $\ucT$ as in \eqref{eq:affmarkedideal}, we define
\begin{description}
\item[\quad$n(\ucT)$]$:= \max n_\al$;
\smallskip
\item[\quad$m(\ucT)$]$:= \dim X_{\ucT}$;
\smallskip
\item[\quad$\mu(\ucT)$]$:= \mu$;
\smallskip
\item[\quad$d(\ucT)$]$:=$ maximum degree of all polynomials in $\ucT$, i.e., of all polynomials in
\begin{equation}\label{eq:invtlist}
\Psi(\ucT) := \{u_{\al\be ,i},\, g_{\al\be ,i},\, f_{\al\be},\, v_{\al_1\be_1\al_2\be_2,1},\, w_{\al_1\be_1\al_2\be_2,1} \};
\end{equation}
\item[\quad$l(\ucT)$]$:=$ maximum over all $(\al, \be),\, \al\in A,\, \be\in B_\al$, of the number of all polynomials listed in
the right-hand side of \eqref{eq:invtlist} for given $(\al, \be)$;
\smallskip
\item[\quad$q(\ucT)$]$:=$ number of neighbourhoods $U_{\al\be}$ in $\ucT$,
i.e., the number of pairs $(\al, \be),\, \al\in A,\, \be\in B_\al$.
\end{description}
\end{definition}

\begin{definition}\label{def:defover}
Let $\ucT$ denote an affine marked ideal \eqref{eq:affmarkedideal}. Let $\ucT'$ denote another affine marked ideal
analogous to $\ucT$, with indices $\al' \in A'$, $\be' \in B'_{\al'}$, and with $\mu(\ucT') = \mu(\ucT)$. We say that $\ucT'$
is \emph{defined over} $\ucT$ if
\begin{enumerate}
\item there are mappings of index sets $p: A' \to A$ and $p_{\al'}: B'_{\al'} \to B_{p(\al')}$;
\item if $\al = p(\al')$, then $n_\al \leq n_{\al'}$ and the morphism $\pi_{\al'}: (\IK^{n_{\al'}})_{\al'} \to (\IK^{n_\al})_\al$
given by projection onto the first $n_\al$ coordinates, determines birational morphisms
$$
\pi_{\al'\be'} = \pi_{\al'}|_{X'_{\al'\be'}}: X'_{\al'\be'} \to X_{\al\be},\quad\text{where }\, \be = p_{\al'}(\be'),
$$
commuting with $i_{\al'_1\be'_1\al'_2\be'_2}$ and $i_{\al_1\be_1\al_2\be_2}$;
\item there is a birational morphism $X'_{\ucT'} \to X_{\ucT}$ commuting with $j_{\al'\be'}$ and $j_{\al\be}$.
\end{enumerate}
\end{definition}

\section{Blowing up}\label{sec:blup}
We give effective estimates for the blowing-up of an affine marked ideal, corresponding the transform
of a marked ideal by an admissible blowing-up.

Let $\ucT$ denote an affine marked ideal \eqref{eq:affmarkedideal}; for example, obtained from a marked
ideal $\ucI$. Write
$$
n = n(\ucT),\quad m = m(\ucT),\quad d = d(\ucT).
$$

\medskip\noindent
\emph{Centre of blowing up.} We assume there is an open subcovering 
$\{U_{\al\be'}\}_{\al \in A, \be' \in B'_{\al}}$ of $(\IK^{n_\al})_\al$, together with mappings of index sets
$p = p_\al: B'_\al \to B_\al$, and closed subvarieties $C_{\al\be'} \subset (\IK^{n_\al})_\al$, of dimension
$k_{\al\be'} \leq m$, such that
\begin{enumerate}
\item $\bigcup_{p_\al(\be')=\be} U_{\al\be'} = U_{\al\be}$;
\smallskip
\item $C_{\al\be'} \cap U_{\al\be'} \subset \cosupp(\cI_{\al,p(\be')},\mu) \cap U_{\al\be'}$;
\item  $C_{\al\be'}$ is defined in  $U_{\al\be'}$ by local parameters
$$
u_{\al\be', 1},\ldots,u_{\al\be',n_\al -m} ,\, u_{\al\be',n_\al -m+1},\ldots, u_{\al\be',n_\al - k_{\al\be'}} \in \IK[x_\al],
$$
where $X_{\al\be}$ is defined in $U_{\al\be'} \subset U_{\al\be}$ by
$$
u_{\al\be', 1},\ldots,u_{\al\be',n_\al -m} ;
$$
\item Each $u_{\al\be', 1},\ldots,u_{\al\be',n_\al - k_{\al\be'}} $ is transverse to $E_{\al\be}$ over $U_{\al\be'}$,
or coincides with an affine coordinate defining an exceptional divisor.
\end{enumerate}

The blowing-up with centre $C = \{C_{\al\be'}\}$ can be described by an affine marked ideal
\begin{equation}\label{eq:affover}
\ucT' = \left(\left\{U'_{\al'\be'}, X'_{\al'\be'}, E'_{\al'\be'}, \cI'_{\al'\be'}, (\IK^{n_\al'})_{\al'}: \al' \in A',\,\be'\in B'_{\al'} \right\}, \mu\right),
\end{equation}
defined over $\ucT$, with the following ingredients:

\medskip\noindent
\emph{Open covering after blowing up.} The blow-up creates a new collection of affine spaces $(\IK^{n_{\al'}})_{\al'}$; more
precisely, we associate to the functions $u_{\al\be',i}$ on $(\IK^{n_\al})_\al$, $i = 1,\ldots, n_\al -k_{\al\be'}$,
$n_\al - k_{\al\be'}$ affine charts $(\IK^{n_{\al'}})_{\al'}$, where $\al' := (\al,i)$, $i = 1,\ldots, n_\al -k_{\al\be'}$, and
$n_{\al'} := 2n_\al - k_{\al\be'}$. We also create a new collection of open sets $U_{\al'\be'} \subset (\IK^{n_{\al'}})_{\al'}$,
by taking the inverse images of $U_{\al\be'} \subset (\IK^{n_\al})_\al$ under the morphisms $\pi_\al: (\IK^{n_{\al'}})_{\al'}
\to (\IK^{n_{\al}})_{\al}$ of projection onto the first $n_\al$ coordinates.

\medskip\noindent
\emph{Equations of blowing up.} The blowing-up in each of the $n_\al -k_{\al\be'}$ affine charts $(\IK^{n_{\al'}})_{\al'}$,
$\al' := (\al,i)$, $i = 1,\ldots, n_\al -k_{\al\be'}$, can be described as follows (where, for simplicity of notation, we drop the indices
$\al,\,\be$). Assume that the function $u_{i_0}$, $i_0 \leq n-k$, determines the chart of the blowing-up. Then the
blowing-up of $\IK^n$ is the closed subset $\Bl(\IK^n)$ of the $u_{i_0}$-chart $\IK^{2n-k}$ given by the equations
\begin{align*}
u_j -u_{i_0} x_{j+n} &= 0, \quad 1\leq j\leq n-k,\quad j\neq i_0,\\
u_{i_o} - x_{i_0 + n} &= 0.
\end{align*}

\medskip\noindent
\emph{Exceptional divisors.} The exceptional divisor of the preceding blowing-up is given by $u_{i_0} = 0$ on 
$\Bl(\IK^n) \subset \IK^{2n-k}$. Since $u_{i_o} = x_{i_0 + n}$, we can represent this exceptional divisor by
the coordinate $x_{i_0 +n}$ of $\IK^{2n-k}$. The previous exceptional divisors keep their form $x_j = 0$ if they were
not among the parameters describing $C$, or they are converted to $x_{j+n} = u_j/u_{i_0}$ if they were described
by $u_j \equiv x_j$.

\medskip\noindent
\emph{Strict transform of $X = X_{\ucT}$.} Recall that $X = X_{\al\be}$ is given by $u_1 = \cdots = u_{n-m} = 0$
on $U = U_{\al\be} \subset \IK^n$. The blow-up (strict transform) of $X$ is a closed subset $X' \subset \IK^{2n-k}$
given by the equations
\begin{align*}
u_j -u_{i_0} x_{j+n} &= 0, \quad 1\leq j\leq n-k,\quad j\neq i_0, \label{eq:sttransf1}\\
u_{i_o} - x_{i_0 + n} &= 0,\\
x_{j+n} &= 0, \quad 1\leq j\leq n-m,\quad j\neq i_0,\\
1 &= 0, \quad \text{if}\quad 1\leq i_0\leq n-m.
\end{align*}
(Note that, if $1\leq i_0\leq n-m$, then the strict transform is empty in the relevant chart.)

\medskip\noindent
\emph{Birational mappings.} The projections $\pi_\al: (\IK^{n_{\al'}})_{\al'}\to (\IK^{n_{\al}})_{\al}$ induce birational morphisms 
$\pi_{\al'\be'} := \pi_{\al}|_{X'_{\al'\be'}}: X'_{\al'\be'} \to X_{\al\be'}$,
for any $\al,\,\be'$ such that $X_{\al\be'} \neq \emptyset$, and we obtain birational mappings
$$
i_{\al'_1\be'_1\al'_2\be'_2}: X'_{\al'_1\be'_1} \xrightarrow{\pi_{\al'_1\be'_1}} X_{\al_1\be'_1} 
\xdashrightarrow{i_{\al_1\be'_1\al_2\be'_2}} X_{\al_2\be'_2} \xdashrightarrow{(\pi_{\al'_2\be'_2})^{-1}} X_{\al'_2\be'_2}.
$$
Moreover, there are open embeddings $j_{\al'\be'}: X_{\al'\be'} \cap U_{\al'\be'} \hookrightarrow X_{\ucT'}$ induced
by $j_{\al\be}: X_{\al\be'} \cap U_{\al\be'} \hookrightarrow X_{\ucT}$, defining an open covering of $X_{\ucT'}$
and satisfying 
$$
(j_{\al'_2\be'_2})^{-1} \circ j_{\al'_1\be'_1} = i_{\al'_1\be'_1\al'_2\be'_2}.
$$

\medskip\noindent
\emph{Generators of $\cI_{\al\be}$ after blowing up.}
We modify the given generators $g_{\al\be,i}$ of $\cI_{\al\be}$ before computing their transforms. We again drop
the indices $\al,\be$ for simplicity of notation. The generators $g_i$ of $\cI$ satisfy the condition
$$
g_i \cdot f^{r_i} \in \cI_C^\mu + \cI_X,
$$
for some $r_i \in \IN$, where $V(f) = \IK^n\backslash U \subset \IK^n$. So we can write
\begin{align*}
g_i\cdot f^{r_i} &= \sum_{a_{n-m+1}+\cdots+a_{n-k} = \mu} h_{i,(a_{n-m+1},\ldots,a_{n-k})} u_{n-m+1}^{a_{n-m+1}}\cdots u_{n-k}^{a_{n-k}}
                                                                                                             + \sum_{j=1}^{n-m} h_{ij}u_j\\
                        &= \sum_{|\oa|=\mu} h_{i\oa} \ou^{\oa} + \sum_{j=1}^{n-m} h_{ij}u_j,
\end{align*}
where $\oa = (a_{n-m+1},\ldots,a_{n-k}),\, \ou^{\oa} = u_{n-m+1}^{a_{n-m+1}}\cdots u_{n-k}^{a_{n-k}}$.

To bound $r_i$ and the degrees $\deg h_{i\oa},\, \deg h_{ij}$, we can consider an expression
$$
g_i\cdot f^{R_i} = \sum_{|\oa|=\mu} H_{i\oa} \ou^{\oa} + \sum_{j=1}^{n-m} H_{ij}u_j,
$$
and introduce a new variable $z$ to get an equation
\begin{align*}
g_i &= z^{R_i} f^{R_i} g_i + (1 - z^{R_i} f^{R_i}) g_i \\
      &= z^{R_i} \left(\sum_{|\oa|=\mu} H_{i\oa} \ou^{\oa} + \sum_{j=1}^{n-m} H_{ij}u_j\right) + g_i \left(\sum_{j=0}^{R_i-1} (zf)^j(1-zf)\right).
\end{align*}
In other words, $g_i$ belongs to the ideal generated by $ \ou^{\oa},\, u_j,\, 1-zf$.

From effective estimates in the ideal membership problem \cite{Asch} (see also \cite{Giusti}, \cite{MM}), it follows that 
we can represent $g_i$ as
$$
g_i = \sum_{|\oa|=\mu} \tH_{i\oa} \ou^{\oa} + \sum_{j=1}^{n-m} \tH_{ij}u_j + \tH(1-zf),
$$
using polynomials $\tH_{i\oa},\,\tH_{ij},\,\tH$ of degrees $\leq (2d\mu)^{2^n}$. We can then substitute $z=1/f$ and clear
denominators to get a bound $d(2d\mu)^{2^{n+1}}$  on $r_i,\,\deg h_{i\oa},\, \deg h_{ij}$.

\medskip\noindent
\emph{Effective estimates on the transform by blowing up}. We describe the (controlled) transform of $\cI = \cI_{\al\be}$
using the modified generators
$$
\og_i =  \sum_{|\oa|=\mu} h_{i\oa} \ou^{\oa}.
$$
The transform of $\og_i$ is given by the pullback of $\og_i$ divided by the exceptional divisor
$x_{i_0 + n}$ to the power $\mu$, in the $u_{i_0}$ chart as described above.

\begin{lemma}\label{lem:blup}
Let $\ucT$ denote an affine marked ideal \eqref{eq:affmarkedideal}. Assume that $d$ is a bound
on both $d(\ucT)$ and the degrees of the polynomials describing the centre of a blowing up, as above. If $\ucT'$
is the affine marked ideal \eqref{eq:affover} defined over $\ucT$ by the blowing up, then $d(\ucT') \leq (2d\mu)^{2^{n+2}}$.
\end{lemma}

To summarize the effect of a blowing-up, or of a sequence of blowings-up:

\begin{lemma}\label{lem:blupsummary}
Suppose $\ucT = \ucT_r$ is an affine marked ideal \eqref{eq:affmarkedideal} corresponding to the transform $\ucI = \ucI_r$
in some year $r$ of the resolution history of a given initial marked ideal $\ucI_0$. Associate to $\ucT$ the vector
\begin{equation*}
\ga = \ga(\ucI) = \ga(\ucT) := (r,n,m,d,l,q,\mu),
\end{equation*}
where $n=n(\ucT),\, m=m(\ucT),\, d=d(\ucT),\, l=l(\ucT),\, q=q(\ucT)$. Consider a blowing-up as above, and assume
that $d$ (respectively, $q$) bounds also the degrees of the polynomials describing the centre of blowing up (respectively,
the number of neighbourhoods $U_{\al\be'}$). If $\ucT' = \ucT_{r+1}$ is the affine marked ideal defined over $\ucT$
by the blowing up, then 
$$
\ga(\ucT') = (r',n',m',d',l',q',\mu'),
$$
where
\begin{equation}\label{eq:blupsummary}
r' = r+1,\,\ n'\leq 2n,\,\ m'=m,\,\  d' \leq (2d\mu)^{2^{n+2}},\,\  l' \leq l+n,\,\  q'\leq nq,\,\  \mu'=\mu.
\end{equation}

In other words, the effect of a single blowing-up is measured by the function
\begin{equation*}
\Bl(r,n,m,d,l,q,\mu) := (r+1\,, 2n,\, m,\, G(n,d,\mu),\,l+n,\,nq,\,\mu),
\end{equation*}
where $G(n,d,\mu) := (2d\mu)^{2^{n+2}}$.

After a sequence of $t$ blowings-up of $\ucT = \ucT_r$, $\ga(\ucT_{r+t})$ is bounded (term by term) by
the recursively defined function
\begin{align*}
\overline{\Bl}(\ga,t) :&= \Bl(\overline{\Bl}(\ga,t-1))\\
&= (r+t,\, 2^tn,\, m,\, G(t, n,d,\mu),\, l+2^{t-1}n,\, (2^{t+1}-1)n^t q,\, \mu),
\end{align*}
where $G(t, n,d,\mu)$ is also determined by the recursive formula.
\end{lemma}

\section{Bounds on the order of an ideal and on the degree of a derivative ideal}\label{sec:bounds}
We continue to use the notation of Sections \ref{sec:affmarkedideal}, \ref{sec:blup}. 
We first give a bound $M(n,d)$ on the order (multiplicity) of $\cI$ on $X$, depending only on $n=n_\al$ and
a bound $d$ on the degrees of the polynomials generating $\cI_X$, where $X = X_{\al\be}$, and
$\cI = \cI_{\al\be}$ in $\IK[x_1,\ldots,x_n] = \IK[x_{\al,1},\ldots,x_{\al,n_\al}]$.

After a linear coordinate transformation, we can assume that the order of each polynomial $u_{i_j} - x_j$ is $\geq 2$,
$j=1,\ldots, n-m$, where $(i_1,\ldots,i_{n-m})$ is a permutation of $(1,\ldots,n-m)$. (Recall that $\cI_X = (u_1,\ldots,u_{n-m})$.)

For any polynomial $g \in \cI$, we can find $h_j \in \IK[x_1,\ldots,x_n]$, $j=0,\ldots,n-m$, and $h\in \IK[x_{n-m+1},\ldots,x_n]$
such that
\begin{equation}\label{eq:nullstell1}
h_0 g + \sum_{j=1}^{n-m}h_j u_{i_j} = h(x_{n-m+1},\ldots,x_n).
\end{equation}
We can write \eqref{eq:nullstell1} over the field of fractions $\IK(x_{n-m+1},\ldots,x_n)$ as
\begin{equation}\label{eq:nullstell2}
\tilh_0 g + \sum_{j=1}^{n-m}\tilh_j u_{i_j} = 1,
\end{equation}
where $\tilh_j = h_j/h \in \IK(x_{n-m+1},\ldots,x_n)[x_1,\ldots,x_{n-m}]$.

The effective Nullstellensatz \cite{Brown}, \cite{Jel}, \cite{Ko} provides a solution $\tilh_j$, $j=0,\ldots,n-m$, of \eqref{eq:nullstell2}
with a bound $d^{n-m}$ on the degrees of the $\tilh_j$ with respect to $x_1,\ldots,x_{n-m}$. We can treat \eqref{eq:nullstell2}
as a system of linear equations with unknowns in $\IK(x_{n-m+1},\ldots,x_n)$ (and given functions in $\IK[x_{n-m+1},\ldots,x_n]$),
by equating coefficients of the monomials in $(x_1,\ldots,x_{n-m})$ on both sides of \eqref{eq:nullstell2}. The number of monomials,
and therefore the number of equations in the linear system, is bounded by 
\begin{equation}\label{eq:choose}
{d^{n-m} + n - m \choose n-m}.
\end{equation}
A solution of this system of linear equations is given by minors of the linear system, so their degrees (as rational functions
in $\IK(x_{n-m+1},\ldots,x_n)$) are also bounded by \eqref{eq:choose}. By clearing denominators, we obtain \eqref{eq:nullstell1}
with $\deg h,\, \deg h_j \leq M(n,d)$, where $M(n,d) = d^{O(n^2)}$ is given by an explicit formula.

\begin{lemma}\label{lem:mult}
The order $\ord_x \cI$, $x\in X$, is bounded by $M(n,d) = d^{O(n^2)}$.
\end{lemma}

\begin{proof}
If $g\in \IK[x_1,\ldots,x_n]$, then 
$$
\ord_x (g|_X) \leq \ord_x (h_0g)|_X = \ord_x h|_X
= \ord_x h \leq \deg h \leq M(n,d).
$$
\end{proof}

\subsection{Derivations on $X$}\label{subsec:der}
Given $X = X_{\al\be} \subset \IK^n$ as above, and $E = E_{\al\be}$, let $\Der_{\IK^n,X,E}$ denote the ring of 
derivations $D$ (sections of the sheaf of derivations) on $\IK^n$ such that $D(\cI_X) \subset \cI_X$ and 
$D(\cI_E) \subset \cI_E$. Then $\Der_{\IK^n,X,E}$ is generated by
\begin{align*}
u_i \p_{u_j}, \quad &1 \leq i,j \leq n-m,\\
u_j \p_{u_j}, \quad &\text{if $u_j = x_j$ is an exceptional divisor},\\
\p_{u_j}, \quad &\text{for the remaining $j$},
\end{align*}
where $\p_u$ denotes $\p/\p u$. (Note that the generator $u_j \p_{u_j}$, in the case that $u_j = x_j$ is an exceptional divisor,
can be replaced by $x_j \p_{x_j}$.)
In particular, if $\Der_{X,E} := \Der_{\IK^n,X,E}|_X$, then $\Der_{X,E}$ is generated by
\begin{align*}
u_j \p_{u_j}, \quad &\text{if $u_j = x_j$ is an exceptional divisor},\\
\p_{u_j}, \quad &\text{otherwise},
\end{align*}
for all $j=n-m+1,\ldots,n$.

Since 
$$
\p_{x_i} = \sum_{j=1}^n \frac{\p u_j}{\p x_i}\cdot \p_{u_j},
$$
the (column) vector $(\p_{u_j})_j$ with components $\p_{u_j}$ can be written
$$
(\p_{u_j})_j = \bigg(\det \left(\frac{\p u_j}{\p x_i}\right)_{ij}\bigg)^{-1} \cdot (\p'_{u_j})_j,\quad \text{where}\quad
(\p'_{u_j})_j = \adj \bigg(\left(\frac{\p u_j}{\p x_i}\right)_{ij}\bigg)\cdot (\p_{x_i})_i,
$$
and $\adj$ denotes the adjugate of a matrix. We obtain the following lemma.

\begin{lemma}\label{lem:der}
Let $V$ denote the open subset of $\IK^n$ where the determinant above does not vanish. Then 
$U = U_{\al\be} \subset V$, and $\Der_{\IK^n,X,E}$ is generated over $U$ by
\begin{align*}
u_i \p'_{u_j}, \quad &1 \leq i,j \leq n-m,\\
x_j \p_{x_j}, \quad &\text{if $u_j = x_j$ is an exceptional divisor},\\
\p'_{u_j}, \quad &\text{for the remaining $u_j$},
\end{align*}
\end{lemma}

The derivations listed in Lemma \ref{lem:der} form a subsheaf $\overline{\Der_{\IK^n,X,E}}$ of $\Der_{\IK^n,X,E}$.
These sheaves coincide on $U$, so we will replace $\Der_{\IK^n,X,E}$ by $\overline{\Der_{\IK^n,X,E}}$ for
computations over $U$.

\begin{lemma}\label{lem:derbounds}
Suppose $\cI \subset \IK[x_1,\ldots,x_n]$ is generated by polynomials of degree $\leq d_1$,
and that $\deg u_i \leq d_2$, for all $i$. Then $\overline{\Der_{\IK^n,X,E}}(\cI)$ is generated by elements
of degree $\leq d_1 + n(d_2 -1)$.
\end{lemma}

\section{Estimates for the companion and coefficient ideals.}\label{sec:compcoeff}
Let $\ucI = (Z,X,E,\cI,\mu)$ denote a marked ideal. We write $\mu(\ucI):=\mu$. Recall that, in the general case II
of resolution of singularities of a marked ideal, we construct the monomial and residual parts $\ucM(\ucI)$
and $\ucR(\ucI)$ of $\ucI$, the companion ideal $\ucG(\ucI)$, and the sum of derivative ideals,
$\ucC(\ucG(\ucI)) := \sum_{j=0}^{k-1} \ucD_E^j(\ucG(\ucI))$, where $k = \mu(\ucG(\ucI))$. 

When we associate an affine marked ideal $\ucT$ \eqref{eq:affmarkedideal} to $\ucI$, $X = X_{\al\be} \subset U_{\al\be} = U \subset \IK^n$
is defined by parameters $u_1,\ldots,u_{n-m}$, and we use $\overline{\Der_{\IK^n,X,E}}$ instead of $\Der_{X,E}$ to define
the derivative ideals $\ucD_E^j(\ucI)$. In particular, we will write $\overline{\Der_{\IK^n,X,E}}^j(\ucI)$ in place of $\ucD_E^j(\ucI)$ in
this context.

Let $\omu := \mu(\ucR(\ucI))$. Then we have the following bounds on orders.

\begin{lemma}\label{lem:ordbounds}
$\omu \leq M(n,d),\ \ \mu(\ucG(\ucI)) \leq \mu\omu,\ \ \mu(\ucC(\ucG(\ucI))) \leq (\mu\omu)! \leq (\mu\cdot M(n,d))!.$
\end{lemma}

The following two lemmas are consequences of Lemmas \ref{lem:derbounds}, \ref{lem:ordbounds}.

\begin{lemma}\label{lem:degbound}
The maximal degree of generators of $\ucC(\ucG(\ucI))$ (i.e., of the affine marked ideal induced by $\ucT$) is bounded by
$$
A(n,d,\mu) := (\mu\omu)! (n+1)d \leq (\mu\cdot M(n,d))!(n+1)d \leq (\mu\cdot d^{O(n^2)})!.
$$
\end{lemma}

\begin{lemma}\label{lem:maxcontactbounds}
The maximal degree of any $u \in \overline{\Der_{\IK^n,X,E}}^{\omu -1}\cR(\ucI)$ defining a hypersurface of
maximal contact is bounded by
$$
B(n,d,\mu) := \omu (n+1)d \leq M(n,d)\cdot (n+1)d \leq d^{O(n^2)}.
$$
\end{lemma}

The following lemma provides bounds on the number of generators of ideals associated to the marked ideal
$\ucI = (Z,X,E,\cI,\mu)$, in terms
of the number of generators $l(\ucI)$ of $\cI$. These estimates are simple consequences of the definitions.
(In the statement of the lemma, the number of generators means the number of generators $l(\cdot)$
of the ideal involved in the indicated marked ideal.)

\begin{lemma}\label{lem:nogens}
Let $l = l(\ucI)$. Then:
\begin{enumerate}
\item The number of generators of $\overline{\Der_{\IK^n,X,E}}^j(\ucI)$ is $(n+1)^j l$.

\smallskip
\item Bounds on the number of generators of the ideals indicated:
\begin{align*}
\ucI^i :&\quad  l^i,\\
\ucG(\ucI) :&\quad  L_{\ucG} (l,\mu) := l^\mu + 1,\\
\ucC(\ucI) :&\quad  L_{\ucC} (l,\mu) := \mu (n+1)^{\mu!}l^{\mu!}.
\end{align*}
\end{enumerate}
\end{lemma}

\begin{corollary}\label{cor:nogens}
The number of generators $l(\ucC(\ucG(\ucI)))$ is bounded by
$$
F(n,d,\mu,l) := L_{\ucC} (L_{\ucG} (l,\mu),\,\mu\cdot M(n,d)),\quad \text{where } l = l(\ucI).
$$
\end{corollary}

In summary:

\begin{lemma}\label{lem:summarybounds}
The effect of passing from $\ucI$ to $\ucC(\ucG(\ucI))$, as in Case IIA of the resolution algorithm,
can be described by the following function of $\ga(\ucI) = (r,n,m,d,l,q,\mu)$:
$$
\De_{\mathrm{IIA}}(r,n,m,d,l,q,\mu) := (r,\, n,\, m,\,A(n,d,\mu),\, F(n,d,\mu,l),\, q,\, (\mu\cdot M(n,d))!).
$$
\end{lemma}

\subsection{Bounds for the number of maximal contact varieties and corresponding neighbourhoods}\label{subsec:maxcontactbounds}
We consider maximal contact hypersurfaces defined in suitable neighbourhoods by functions 
$u \in \overline{\Der_{\IK^n,X,E}}^{\omu -1}(\cR(\ucI))$ of the form $u = \p^a g_i$, 
where the $g_i$ are generators of $\cR(\ucI)$, and $\p^a  = \p_1^{a_1}\cdots \p_n^{a_n}$, with $\p_1,\ldots,\p_n$ derivations from 
the list in Lemma \ref{lem:der}, and $|a| = \omu -1$, $|a| := a_1 +\cdots a_n$.

More precisely, for each $i$ and each $b=(b_1,\ldots,b_n)$ with $|b|=\omu$, we consider the open set $U_{b,i} 
= U\backslash V(\p^b g_i)$, and the maximal contact subvarieties defined by $\p^a g_i$, where $a$ is obtained from $b$ by
replacing some $b_j >0$ by $b_j -1$. 

The number of such $\p^a g_i$ is bounded by ${\omu +n \choose n}l(\ucI)$, so we get the following.

\begin{lemma}\label{lem:maxcontactbounds}
The number of maximal contact hypersurfaces given by the $\p^a g_i \in \overline{\Der_{\IK^n,X,E}}^{\omu -1}(\cR(\ucI))$,
and the number of neighbourhoods $U_{b,i}\subset U$ are each bounded by $C(n,d,\mu)\cdot l(\ucI)$, where
$$
C(n,d,\mu) := {M(n,d) + n \choose n}.
$$
\end{lemma}

\begin{lemma}\label{lem:summarymaxcontactbounds}
The effect of passing from $\ucI$ to $\ucC(\ucG(\ucI))|_{V(u)}$, where $u = \p^a g_i$ as above, in Case I
of the algorithm, can be described by the following function of $\ga(\ucI) = (r,n,m,d,l,q,\mu)$:
$$
\De_{\mathrm{I}}(r,n,m,d,l,q,\mu) := (r,\, n,\, m-1,\,A(n,d,\mu),\, F(n,d,\mu,l),\, q\cdot C(n,d,\mu),\, (\mu\cdot M(n,d))!).
$$
\end{lemma}

\section{Recursive desingularization estimates}\label{sec:summaryest}
Let $\ucT = \ucT^{(m)}$ denote an affine marked ideal (Definition \ref{def:affmarkedideal}) with associate data vector
$\ga = (r,n,m,d,l,q,\mu)$ (Lemma \ref{lem:blupsummary}), corresponding to a
marked ideal $\ucI = (Z,X,E,\cI,\mu)$ of dimension $m$ (i.e., $\dim X = m$). (We are allowing $r>0$ to consider a marked ideal that
arises after a sequence of $r$ blowings-up.) We introduce a function of $\ga$ defined as follows.

Let $\ucT_*$ denote the affine marked ideal defined over $\ucT$ given by functorial resolution of singularities.
Define
\begin{equation*}\label{eq:recfn}
\Ga^{(m)}(\ga) := (r+R^{(m)}(\ga),\, N^{(m)}(\ga),\, m,\, D^{(m)}(\ga),\,L^{(m)}(\ga),\, Q^{(m)}(\ga),\,\mu),
\end{equation*}
where:
\smallskip
\begin{description}
\item[\quad$R^{(m)}(\ga)$] is a bound on the number of blowings-up needed to resolve $\ucI$ or $\ucT$;
\smallskip
\item[\quad$N^{(m)}(\ga)$] is a bound on the dimensions of all ambient spaces constructed in the resolution process;
\smallskip
\item[\quad$D^{(m)}(\ga)$] is a bound on the maximum of the degrees of all polynomials in the description of $\ucT_*$
and in all objects constructed in the resolution process;\\[-.85em]
\item[\quad$L^{(m)}(\ga)$] is a bound on the number of polynomials appearing in the description of a single neighbourhood
$U_{\al\be}$ constructed in the process of resolving $\ucT$;
\smallskip
\item[\quad$Q^{(m)}(\ga)$] is a bound on the number of neighbourhoods in all auxiliary objects (e.g., centres of blowing up)
appearing in the resolution process.
\end{description}

\begin{remark}\label{rem:recfn}
The functions $R^{(m)}(\ga),\, N^{(m)}(\ga),\, D^{(m)}(\ga)$ do not depend on $l,\,q$.
\end{remark}

\subsection{Desingularization algorithm}  We begin with an affine marked ideal $\ucT^{(m)}$ with initial data
$\ga = (0,n,m,d,l,q,\mu)$, associated to a marked ideal $\ucI = (Z,X,E,\cI,\mu)$ of dimension $m$.

Our aim is give recursive formulas or effective estimates for $R^{(m)}(\ga),\, N^{(m)}(\ga),\,\allowbreak
D^{(m)}(\ga),\, L^{(m)}(\ga),\,
Q^{(m)}(\ga)$, by induction on $m$.

\medskip\noindent
\emph{Base case} $m=0$. $R^{(0)}(\ga) = 1$, $N^{(0)}(\ga) = 2n$, $D^{(0)}(\ga) = O(dn)$, $L^{(0)}(\ga) = l\cdot (dn)^{O(n)}$,\,
$Q^{(0)}(\ga) = nq$.

\medskip\noindent
\emph{Inductive step.} If $\cI = 0$, then we resolve singularities by a single blowing-up with centre $C=X$. In this case,
$\ucT^{(m)}$ is transformed to $\ucT_*^{(m)}$ with $X_* = \emptyset$ and data bounded as in \eqref{eq:blupsummary}.

\medskip
Suppose $\cI \neq 0$. The desingularization algorithm proceeds as follows. (We refer to the Steps in the inductive
proof in Section \ref{sec:algorithm}.)

\medskip\noindent
{\bf Step II.} Let $\omu$ denote the maximum order $\ord\, \cR(\ucI)$ of the residual ideal on $\cosupp \ucI$.
Then $\omu \leq M(n,d)$ (Lemma \ref{lem:ordbounds}). Our aim is to reduce $\omu$ to zero; i.e., to reduce
the monomial case Step IIB.

\medskip\noindent
{\bf Step IIA.} If $\omu > 0$, then we decrease $\omu$ by resolving the singularities
of the companion ideal $\ucJ = \ucG(\ucI)$. The companion ideal corresponds to a new affine marked ideal $\ucT_1^{(m)}$
(see Section \ref{sec:compcoeff}).
The companion ideal $\ucJ$ has maximum order, so resolution is accomplished by Step I of the algorithm.

\medskip\noindent
{\bf Step I.} Resolution of singularities of $\ucJ = \ucG(\ucI))$ (i.e., of $\ucT_1^{(m)}$) is achieved by induction,
by resolving the singularities of the coefficient ideal plus boundary, on a maximal contact hypersurface
$Y = V(u)$ for $\ucJ_\emptyset$ (obtained from $\ucJ$ by deleting $E$).

When we pass from $\ucI$ to the coefficient ideal $\ucC_Y(\ucJ_\emptyset)$ plus boundary, we adjoin equations
for the maximal contact hypersurface and create new neighbourhoods. The effect of passing from $\ucI$ to
$\ucC_Y(\ucJ_\emptyset)$ is described by the function $\De_\mathrm{I}(\ga)$ of 
Lemma \ref{lem:summarymaxcontactbounds}.

\medskip\noindent
{\bf Step IB.} Let $s$ denote the maximum number of components of $E$ at points of $\cosupp \ucJ$.
Passage from $\ucC_Y(\ucJ_\emptyset)$ to the coefficient ideal plus boundary $\ucC_Y(\ucJ_\emptyset) + \ucE^{(s)}$
does not change $\De_\mathrm{I}(\ga)$.

To resolve $\ucC_Y(\ucJ_\emptyset) + \ucE^{(s)}$, we construct a new affine marked ideal $\ucT_2^{(m-1)}$ with
initial data bounded by $\ga^{(\mathrm{I})} := \De_\mathrm{I}(\ga)$. By induction, resolution of $\ucT_2^{(m-1)}$
requires at most $R^{(m-1)}(\ga^{(\mathrm{I})})$ blowings-up. The maximum degree of the polynomials involved
in the resolution (including description of the centres) is bounded by $D^{(m-1)}(\ga^{(\mathrm{I})})$, and the
dimension of the affine charts involved is bounded by $N^{(m-1)}(\ga^{(\mathrm{I})})$.

The resolution $\ucT_{2*}^{(m-1)}$ of $\ucT_2^{(m-1)}$ determines a composite of blowings-up
$\ucT_{1*}^{(m)}$ of $\ucT_1^{(m)}$, consisting of at most $R^{(m-1)}(\ga^{(\mathrm{I})})$ blowings-up. 
The data bound
$$
\Ga^{(m-1)}(\De_\mathrm{I}(\ga))
$$ 
for $\ucT_{2*}^{(m-1)}$ given by induction, remains valid for $\ucT_{1*}^{(m)}$, since
we use the same centres, same ambient affine spaces, etc. Only the objects $Y,\,X$ (and their transforms)
are different; this does not affect the bounds for the data---we have additional equations in $\ucT_{2*}^{(m-1)}$
as compared to $\ucT_{1*}^{(m)}$.

Step IB is performed at most $m$ times (with $s = m-1,\ldots,0$; the case $s=0$ is Step IA in Section \ref{sec:algorithm}). 
We can introduce a counter $t=0,\ldots,m$,
and a function $\Ga_{\mathrm{IB}}^{(m)}(\ga,t)$ which measures the possible effect of performing Step IB $t$ times:
\begin{align*}
\Ga_{\mathrm{IB}}^{(m)}(\ga,0) &:= \De_\mathrm{I}(\ga),\\
\Ga_{\mathrm{IB}}^{(m)}(\ga,t)  &:= \Ga^{(m-1)}(\Ga_{\mathrm{IB}}^{(m)}(\ga,t-1)),\quad t=1,\ldots,m.
\end{align*}

Step IB is performed at most $m$ times (the last time with $s=0$ to complete Step I). The resolution process
for the coefficient ideal plus boundary leads to resolution of the companion ideal $\ucJ$ or the corresponding
object $\ucT_1^{(m)}$, with data bounded by 
$$
\Ga_{\mathrm{I}}^{(m)}(\ga) := \Ga_{\mathrm{IB}}^{(m)}(\ga,m) = (r^{(\mathrm{I})}, n^{(\mathrm{I})}, m, d^{(\mathrm{I})},
l^{(\mathrm{I})}, q^{(\mathrm{I})}, (\mu\cdot M(n,d))!),
$$
with the appropriate entries $r^{(\mathrm{I})}, n^{(\mathrm{I})}, d^{(\mathrm{I})},
l^{(\mathrm{I})}, q^{(\mathrm{I})}$ coming from Lemmas \ref{lem:blupsummary}, \ref{lem:summarymaxcontactbounds}
and Corollary \ref{cor:nogens}.
This completes Step I.

\medskip
The effect of Step I is to transform the affine marked ideal $\ucT^{(m)}$ correponding to $\ucI$, with
initial data $\ga$, to a new object with data bounded by
$$
\Ga_{\mathrm{IIA}}^{(m)}(\ga) := (r^{(\mathrm{I})}, n^{(\mathrm{I})}, m, d^{(\mathrm{I})},
l^{(\mathrm{I})}, q^{(\mathrm{I})}, \omu),
$$
with smaller $\omu < M(n,d)$ (the maximum order of the new residual ideal. This completes Step IIA.

Step IIA is repeated at most $M(n,d)$ times, until the maximum order of the residual ideal drops to $0$,
and we arrive at the monomial case. The final effect of Step IIA is measured by a recursive function
\begin{align*}
\Ga_{\mathrm{IIA}}^{(m)}(\ga,0) &:= \ga,\\
\Ga_{\mathrm{IIA}}^{(m)}(\ga,t)  &:= \Ga_{\mathrm{IIA}}^{(m)}(\Ga_{\mathrm{IIA}}^{(m)}(\ga,t-1)),
\end{align*}
with estimates coming from Lemmas \ref{lem:blupsummary} and \ref{lem:mult}.

Putting $t=M(n,d)$ gives the final effect after completing all necessary steps IIA, and then passing to IIB:
\begin{equation*}\label{eq:recursion}
\Ga_{\mathrm{IIB}}^{(m)}(\ga) := \Ga_{\mathrm{IIA}}^{(m)}(\ga,M(n,d)).
\end{equation*}
At this point, we have reduced $\ucI$ to a monomial marked ideal $\ucI = \ucM(\ucI)$.

\medskip\noindent
{\bf Step IIB.} Desingularization of the monomial marked ideal $\ucI = \ucM(\ucI)$. This marked ideal
corresponds to an affine marked ideal $\ucT^{(m)}$ with data bounded by
$$
\Ga_{\mathrm{IIB}}^{(m)}(\ga) = (r^{(\mathrm{IIB})}, n^{(\mathrm{IIB})}, m, d^{(\mathrm{IIB})},
l^{(\mathrm{IIB})}, q^{(\mathrm{IIB})}, \mu).
$$
The resolution of $\ucI = (x^\al)$ is given by a sequence of blowings-up, each of which decreases
the order $|\al| \leq d^{(\mathrm{IIB})}$. This resolution requires at most $d^{(\mathrm{IIB})}$ blowings-up,
and the final resolution data vector is bounded by the function 
$$
\Ga^{(m)}(\ga) := \overline{\Bl}(\Ga_{\mathrm{IIB}}^{(m)}(\ga),\, d^{(\mathrm{IIB})}),
$$
from Lemma \ref{lem:blupsummary}.

\begin{remark}[\bf Summary remark]\label{rem:summaryremark}
The data involved in resolution of singularities of a marked ideal $\ucI = (Z,X,E,\cI,\mu)$
with initial data $\ga = (0,n,m,d,l,q,\mu)$, have bounds given by the recursive functions above, with final data
vector majorized by $\Ga^{(m)}(\ga)$.

As in \cite{BGMW}, we can also keep track of the Grzegorczyk complexity classes of the resolution data.
Beginning with Step II above, $M(n,d) \in \cE^3$ (cf. Lemma \ref{lem:mult}; we recall that, in general, $\cE^1$ contains
all linear functions, $\cE^2$ all polynomials, and $\cE^3$ all towers of exponential functions$,\ldots$). In Step IB, then $\De_1(\ga)
\in \cE^3$ and, by induction on $m$, $\Ga^{(m-1)}(\De_\mathrm{I}(\ga)) \in \cE^{m+2}$, and $\Ga_{\mathrm{I}}^{(m)}(\ga) \in \cE^{m+2}$.
In Step IIA, $\Ga_{\mathrm{IIB}}^{(m)}(\ga) \in \cE^{m+3}$, from Lemmas \ref{lem:blupsummary} and \ref{lem:mult}, 
and finally $\Ga^{(m)}(\ga) \in \cE^{m+3}$,
in Step IIB. For more details, see \cite{BGMW}.
\end{remark}

\bibliographystyle{amsplain}

\end{document}